\documentclass[11pt]{amsart}

\newtheorem{theorem}{Theorem}[section]
\newtheorem{lemma}[theorem]{Lemma}
\newtheorem{corollary}[theorem]{Corollary}
\newtheorem{proposition}[theorem]{Proposition}

\theoremstyle{definition}

\usepackage{tikz-cd}
\usepackage{mathrsfs}
\usepackage[all,cmtip]{xy}

\newcommand{\C}{\mathbb{C}}

\newcommand{\D}{\mathbb{D}}

\newcommand{\sD}{\mathscr{D}}
\newcommand{\mD}{\mu_{\mathscr{D}_n}}
\newcommand{\Cn}{\mathbb{C}^n}

\newcommand{\autd}{Aut (\mathbb{D})}

\theoremstyle{remark}

\numberwithin{equation}{section}

\begin{document}

\title{Pick interpolation and invariant functions}

%    author one information
\author{Anindya Biswas}
\address{Department of Mathematics and Statistics, Masaryk University, Brno}
\address{ORCID: 0000-0002-7805-9446}
\email{biswas@math.muni.cz}
\thanks{}

%    author two information
%\author{Later}
%\address{Later}
%\email{Later}
%\thanks{}

%\author{Later}
%\address{Later}
%\email{Later}
%\thanks{}



\date{}

\dedicatory{}

\begin{abstract}

In this article, we establish a connection between Pick bodies and invariant functions. We demonstrate that an invariant function can be associated with any Pick body, which determines the solvability of a given Pick interpolation problem and serves as a generalization of the Carathéodory pseudodistance. A complete description of this invariant function is provided for the open unit disc, and it is shown that it leads to another invariant function that can be regarded as a generalized Lempert function. It is also proved that these two invariant functions are equal if certain geodesics can be found. Lastly, we show that, in a very special case, a result analogous to Lempert's theorem holds for the bidisc and the tridisc.

\end{abstract}

\maketitle

\section{Introduction} % use lowercase except for proper names
Suppose that $\Omega$ is a Carath\'eodory hyperbolic domain in the Euclidean space $\mathbb{C}^m$ for some positive integer $m$. We consider mutually distinct $n$ points $z_1,\ldots, z_n$ in $\Omega$ and $n$ arbitrary points in the open unit disc $\D$. The Pick interpolation problem asks for a necessary and sufficient condition for the existence of a holomorphic function $f:\Omega\rightarrow\D$ satisfying $f(z_j)=w_j,a\leq j\leq n$. When $\Omega=\D$, the problem is independently solved by Pick in 1916 (\cite{Pick1916}), and Nevanlinna in 1919(\cite{Nevanlinna1919}). Their result says that the problem $\D\ni z_j\mapsto w_j\in \D$ is solvable if and only if the (Pick-)matrix 
\begin{align}\label{PickMatrix}
	\begin{pmatrix}
		\frac{1-w_i\overline{w_j}}{1-z_i\overline{z_j}}
	\end{pmatrix}_{1\leq i,j\leq n} 
\end{align}
is positive semi-definite. Moreover, $rank\begin{pmatrix}
	\frac{1-w_i\overline{w_j}}{1-z_i\overline{z_j}}
\end{pmatrix} \leq n-1$ if and only if the interpolation problem has a unique solution and the solution is a Blaschke product of degree $rank\begin{pmatrix}
\frac{1-w_i\overline{w_j}}{1-z_i\overline{z_j}}
\end{pmatrix}$. One of the most elementary proofs of this result can be found in \cite{Marshall}. Interested reader may also consult the book by Agler and McCarthy (\cite{A-M}) for a detailed account of this problem and of advancements in this topic.

A huge part of the research in this direction is influenced by operator theory which is apparent by the account in \cite{A-M} and the references therein. Using operator theoretic techniques, Agler and McCarthy found a framework that gives valuable information on Pick interpolation problem on the bidisc $\D^2$ (\cite{A-M-1999}). The importance of their work is realized by the fact that their treatment can be applied to other domains like the symmetrized bidisc (\cite{B-S}) and annulus (\cite{D-M}). However, due to the abstractness of their treatment, it does not solve the Pick interpolation problem fully. A partial solution related to these ideas may be found in \cite{A-M-2000}.

Though influenced by operator theory, a different approach was pursued by Cole, Lewis and Wermer (\cite{CLW1992}, \cite{CLW1993}, \cite{CW1993}, \cite{CW1996}, \cite{CW1997}). They mainly used the notion of IQ-algebras (certain quotient algebras, see \cite{Bon-Dun}) and properties of Hardy-Hilbert space on $\D$ to study the problem.The object of their study is known as the Pick body or the interpolation body. Apart from characterizing the Pick bodies on $\D$, their works give additional information about Agler and McCarthy's work on the bidisc. 

Given nodes $z_1,\ldots,z_n\in \D$, the functions 
\begin{align*}
	\varphi_j (\lambda)=\prod_{i\neq j}\frac{z_i -\lambda}{1-\overline{z_i}\lambda}
\end{align*}
are one of the building blocks in Cole, Lewis and Wermer's approach through IQ-algebras. It is apparent from their work that these functions and their generalization have fundamental role in the study of Pick interpolation problem. It turns out these functions are also an important object of study in the topic of invariant functions (\cite{Coman2000}, \cite{J-P-Invariant}). We will come back to it later.

A moment's thought will reveal that holomorphic invariance has great implications in Pick interpolation problem. For instance, by the Riemann mapping theorem, solving an interpolation problem in $\mathcal{O}(\D,\D)$ (space of all holomorphic functions from $\D$ to $\D$) is equivalent to solving an interpolation problem in $\mathcal{O}(D_1,D_2)$  where $D_1$ and $D_2$ are simply connected domains in the plane $\C$. Apart from this, we see that the two-point Pick interpolation problem is completely solved for a number of domains using the notion of Carath\'eodory pseudodistance $c^*$. For a Carath\'eodory hyperbolic domain $\Omega\subset \C^m$, $z_1,z_2\in \Omega$, and $w_1,w_2\in \D$, there exists an $f\in \mathcal{O}(\Omega,\D)$ satisfying $f(z_j)=w_j\,\, (j=1,2)$ if and only if 

\begin{align}\label{TwoPointProblem}
	m(w_1,w_2)=\Big|\frac{w_1 -w_2}{1-\overline{w_1}w_2}\Big|\leq c^*_\Omega(z_1,z_2)
\end{align}. This holds because of the definition 
\begin{align*}
	c^* _\Omega (z_1,z_2)=sup\{m(g(z_1),g(z_2)):g\in \mathcal{O}(\Omega,\D)\}.
\end{align*}
For domains like the polydisc $\D^m$, the Euclidean ball $\mathbb{B}_m$ and the symmetrized bidisc $\mathbb{G}_2\subset \C^2$, $c^*$ is known (see \cite{J-P-Invariant}) and hence the two point interpolation problem is solved for these domains. For more than two points, there is a generalization of $c^*$ (\cite{Coman2000}, \cite{J-P-Invariant}) which is being studied for the last few decades. Unfortunately, its role in Pick interpolation is not significant yet, though the framework is quite interesting. Our work in a way generalizes this idea.

A notion which is quite useful in the study of both interpolation problems and invariant functions is known as $n$-geodesic (\cite{Agler-Lykova-Young-2013}). Its importance can be realized from Lempert's theorem where one studies the $2$-geodesics, and from recent developments on three-point interpolation problem. Apart from the disc, the three-point interpolation problem is completely solved for the polydisc and the Euclidean ball (\cite{KosinskiPoly}, \cite{Kosinski-Zwonek-Ball}), and in both cases, it is seen that the notion of geodesics is indispensable. We will see that the notion of $n$-geodesic plays a role in our setting as well.

The main part of this article is divided into three sections. In Section \ref{Decription_of_d} we describe two invariant functions associated with Pick bodies on Carath\'eodory hyperbolic domains. When the functions are considered on the open unit disc, they coincide and we discuss its values in Section \ref{Values_Of_d_For_D}. In Section \ref{D2_D3_description} we show that a result similar to Lempert's theorem holds for a very special case.

\section{Description of the invariant functions}\label{Decription_of_d}
For an arbitrary Carath\'eodory hyperbolic domain $\Omega$ in some $\C^m$, we take $n$ distinct points $z_1,\ldots,z_n\in \Omega$, and we denote by $B^\infty_1$ the interior of $\mathcal{O}(\Omega, \D)$, that is, the set of all bounded holomorphic functions on $\Omega$ with norm strictly less than one. We now construct the following set
\begin{align}\label{OpenPickBody}
	\mathscr{D}_\Omega (z_1,\ldots,z_n)=\{(f(z_1)\ldots,f(z_n)):f\in B^\infty_1\}.
\end{align}
We will often write $\mathscr{D}_n$ for $\mathscr{D}_\Omega(z_1,\ldots,z_n)$ when it is clear from the context. Note that $$\overline{\mathscr{D}_n}=\{(f(z_1)\ldots,f(z_n)):f\in \mathcal{O}(\Omega, \overline{\D})\},$$
and the boundary of $\mathscr{D}_n$, denoted by $\partial \mathscr{D}_n$, consists of the points that can not be interpolated by functions with norm strictly less than one.

Let us begin with describing a few properties of the set $\sD_n$. The first one is trivial

\begin{proposition}\label{DnIsConvexOpen}
	$\mathscr{D}_n$ is an open balanced convex subset of the unit polydisc $\D^n$.
\end{proposition}

Let us recall a couple of facts about the Minkowski functional $\mu_{\mathscr{D}_n}$ of $\mathscr{D}_n$ (\cite{J-P-Invariant}, Chapter 2).
\begin{enumerate}
	\item $\mu_{\mathscr{D}_n}$ is a norm on $\mathbb{C}^n$.
	\item The boundary of $\mathscr{D}_n$ is given by $\partial \mathscr{D}_n=\{X\in \mathbb{C}^n: \mu_{\mathscr{D}_n} (X)=1\}$.
\end{enumerate}

Our next result gives a description of $\mu_{\mathscr{D}_n}$ in terms of the Banach algebra $H^\infty (\Omega)$ of all bounded holomorphic functions on $\Omega$.

\begin{theorem}
	For any $\underline{w}=(w_1,\ldots,w_n)\in \Cn$, $$\mD(\underline{w})=inf \{||g||:g\in H^\infty (\Omega), g(z_j)=w_j, j=1,\ldots,n\}.$$
\end{theorem}
\begin{proof}
	Observe that for any $\underline{w}\in \Cn$, there is a $g\in H^\infty (\Omega)$ such that $w_j=g(z_j)$ for all $j=1,\ldots,n$. For such a $g$ we have $\frac{1}{||g||} \underline{w}\in \overline{\sD _n}$. And hence, we have $\mD (\underline{w})\leq ||g||$ for all such $g$. 
	
	Let $l=\inf \{||g||:g\in H^\infty (\Omega), g(z_j)=w_j, j=1,\ldots,n\}$ and suppose that $\mD (\underline{w})<l$. Then $\mD (\frac{1}{l}\underline{w})<1$ and hence $\frac{1}{l} \underline{w}\in \sD_n$. By definition, there is an $f\in B^\infty_1$ such that $\frac{w_j}{l}=f(z_j)$ So $w_j =lf(z_j)$ for all $j$, $lf\in H^\infty (\Omega)$ and $||lf||<l$. This contradicts our assumption. Hence $\mD (\underline{w})=l$ and this concludes the proof.
\end{proof}

	Let $I=I(z_1,\ldots,z_n)$ be the ideal in the Banach algebra $H^\infty (\Omega)$ consisting of the functions that vanish at each $z_j, 1\leq j\leq n$. The result above says that the normed spaces $(\Cn,\mD)$ and $(H^\infty (\Omega)/I,||\cdot||_q)$ are isometrically isomorphic, where $||\cdot||_q$ is the quotient norm. Moreover, $\mD$ is a Banach algebra norm (i.e., sub-multiplicative) on $\Cn$ with component-wise product. Using Montel's theorem, it is not hard to see that $\mD$ is attained, that is, for a given $\underline{w} \in \Cn$, there exists a $g\in \mathcal{O}(\D,\D)$ such that $g(z_j)=w_j$ for all $j$ and $\mD (\underline{w})=||g||$.

We now want to focus on an old domain introduced by H. Cartan \cite{Cartan1932} and later studied by several mathematicians in different context and form (see \cite{BiswasMaitraBidisc}, \cite{CLW1992}, \cite{Isaev23}). The domain is given by 
\begin{align}\label{D2}
	\D^2 _r=\Big\{(w_1,w_2)\in\D^2:m(w_1,w_2)=\Big|\frac{w_1 -w_2}{1-\overline{w_1}w_2}\Big|<r\Big\},r\in (0,1).
\end{align}
If $\autd$ denotes the group of automorphisms of $\D$, then the connected identity component of $Aut(\D^2 _r)$ is the collection of the maps $\Phi_\varphi:(w_1,w_2)\mapsto(\varphi(w_1),\varphi(w_2)),\varphi\in \autd$ (see \cite{BiswasMaitraBidisc}). This idea can be applied to $\sD_n$ in the following way: we consider a $\varphi\in \autd$ and define $\Phi_\varphi:\mathbb{D}^n \rightarrow\mathbb{D}^n$ as
\begin{align}\label{AutosOfDnAssoctdWithAutD}
	\Phi_\varphi (w_1,\ldots,w_n)=(\varphi(w_1),\ldots,\varphi(w_n)).
\end{align}
Since $f\in B^\infty_1$ if and only if $\varphi\circ f\in B^\infty _1$, we conclude that the restriction $\Phi_\varphi |_{\mathscr{D}_n}$ is an automorphism of $\mathscr{D}_n$.

Now we will see that the domains $\sD_n$ can lead to an invariant function. First we see how $\sD_2$ is related to the Carath\'eodory pseudodistance $c^*$.
\begin{theorem}\label{D2Cara}
	For two distinct points $z_1,z_2\in \Omega$ we have
	\begin{align}\label{D2}
		\sD_\Omega (z_1,z_2)=\sD_2=\D^2 _{c^*_\Omega (z_1,z_2)}.
	\end{align}
\end{theorem}

\begin{proof}
	First we observe that whenever $w_1\neq w_2$ and $w_j, tw_j\in\D$, we have $m(w_1,w_2)\leq m(tw_1,tw_2)$ if an only if $0\leq (|t|^2-1)(1-|t|^2|w_1 w_2|^2)$ if and only if $1\leq |t|$. The rest of the proof easily follows from the fact that $c^*_\Omega (z_1,z_2)$ is always attained by functions of norm one.
\end{proof}

	This gives us that for any $r\in (0,1)$, $\D^2 _r$ is convex. Note that, the boundary of $\D^2 _r$ can be given by 
\begin{align*}
	\partial \D^2 _r=\{(e^{i\theta},e^{i\theta}):\theta\in \mathbb{R}\}\cup\{(w_1,w_2)\in \D^2:m(w_1,w_2)=r\}.
\end{align*}
We observe that the Carath\'eodory pseudodistance is represented by \textit{a certain part} of the boundary of the balanced domain (\ref{D2}).

When $n\geq 3$, there is no analogous description of $\sD_n$. To address this issue, let us consider the following procedure: let $\underline{\alpha}=(\alpha_1,\ldots,\alpha_n)\in \Cn-\{\mathbf{0}\}$. For this $\underline{\alpha}$ and $z_1,\ldots,z_n\in \Omega$, we define
\begin{align}\label{GeneralizedInvariant}
	d^{\Omega}  _{(\underline{z}, {\underline{\alpha} })} (z_i,z_j)=d_{\underline{\alpha} }(z_i,z_j)=sup \{m(f(z_i),f(z_j)):f\in\mathcal{O}(\Omega, \D),\underline{f}\in \C\cdot \underline{\alpha}\}
\end{align}
where $\underline{z}=(z_1,\ldots,z_n)\in \Omega^n$, $\underline{f}=(f(z_1),\ldots,f(z_n))\in\D^n$ and $\C\cdot \underline{\alpha}$ denotes the one dimensional subspace generated by the nonzero element $ \underline{\alpha}$. Note that since $\mathbf{0}\in  \sD_n\cap \C\cdot \underline{\alpha}$ and $\mathcal{O}(\Omega, \D)$ contains the zero function, the set we used to define $d$ is nonempty, and hence the supremum exists. Thus the definition \ref{GeneralizedInvariant} does not require $z_1,\ldots,z_n$ to be mutually distinct, but will always assume it. We will often write $d_{\underline{\alpha} }$ instead of $d^{\Omega}  _{(\underline{z}, {\underline{\alpha} })}$ when there is no scope of confusion.
\begin{proposition}\label{PropsOfd}
Suppose that $z_1,\ldots,z_n$ are mutually distinct points in $\Omega$, and $\underline{\alpha}=(\alpha_1,\ldots,\alpha_{n})\in \C ^n -\{\mathbf{0}\}$.
\begin{enumerate}
	\item $0\leq d_{\underline{\alpha} }(z_i,z_j)\leq c^*_\Omega (z_i,z_j)$
	for all $i$ and $j$.
	\item  $d_{\underline{\alpha} }(z_i,z_j)$ is contractible under holomorphic maps, that is, if $F:\Omega_1\rightarrow\Omega_2$ is a holomorphic map between two Carath\'eodory hyperbolic domains, then
	\begin{align*}
		d^{\Omega_2} _{(F(\underline{z}) , {\underline{\alpha} })}(F(z_i),F(z_j))\leq d^{\Omega_1} _{(\underline{z}, {\underline{\alpha} })}(z_i,z_j)
	\end{align*}
	where $\underline{z}=(z_1,\ldots,z_n)\in \Omega_1 ^n$ and $F(\underline{z})=(F(z_1),\ldots,F(z_n))\in \Omega_2 ^n$.
	\item For any $\underline{\beta}\in \C\cdot \underline{\alpha}$ and $\underline{\beta}\neq \mathbf{0}$, $d_{\underline{\alpha} }(z_i,z_j)=d_{\underline{\beta} }(z_i,z_j)$. Furthermore, if $\alpha_i =\alpha_j$, then $d_{\underline{\alpha} }(z_i,z_j)=0$.
	\item When $\Omega=\D$, we have $d_{\underline{z}}(z_i,z_j)=m(z_i,z_j)$ for all $i$ and $j$.
	\item\label{RelationMinkoInvariant} For any $i$ and $j$, $d_{\underline{\alpha} }(z_i,z_j)=m(t\alpha_i,t\alpha_j)$ where $t=\frac{1}{\mu_{\mathscr{D}_n}(\underline{\alpha})}$. Consequently, for any $\underline{\alpha}$ there is an $f\in \mathcal{O}(\Omega, \overline{\D})$ such that $\underline{f}\in \C\cdot\underline{\alpha}$ and $d_{\underline{\alpha} }(z_i,z_j)=m(f(z_i),f(z_j))$ for all $i$ and $j$.
	\item If we consider only two points $z_1,z_2\in \Omega$ and an $\underline{\alpha}=(\alpha_1,\alpha_2)$ with $\alpha_1\neq\alpha_2$, then $d_{\underline{\alpha} }(z_1,z_2)=c^*_\Omega (z_1,z_2)$.
\end{enumerate}
\end{proposition}
\begin{proof}
	We will prove only (\ref{RelationMinkoInvariant}) and leave the rest for the reader. To see (\ref{RelationMinkoInvariant}), note that $\mu_{\mathscr{D}_n}(t\underline{\alpha})=1$. If $\alpha_i=\alpha_j$, the result is clear. Therefore, let us assume $\alpha_i\neq\alpha_j$, and hence $t\underline{\alpha}\in \partial\sD_n\cap \D^n$. By definition of $\overline{\mathscr{D}_n}$, there is an $f\in \mathcal{O}(\Omega, \D)$ such that $t\underline{\alpha}=\underline{f}$. Also suppose that $g\in  \mathcal{O}(\Omega, \D)$ satisfies $d_{\underline{\alpha} }(z_i,z_j)=m(g(z_i),g(z_j))$ and $\underline{g}=\lambda\underline{\alpha}$ for some $\lambda\in \C$. We now have 
	\begin{align*}
		m(t\alpha_i,t\alpha_j)&=m(f(z_i),f(z_j))\\
		&\leq d_{\underline{\alpha} }(z_i,z_j)\\
		&=m(g(z_i),g(z_j))\\
		&=m(\lambda \alpha_i,\lambda \alpha_j).
	\end{align*}
	So $t\leq |\lambda|$. On the other hand $\frac{|\lambda|}{t}=|\lambda|\mD (\underline{\alpha})=\mD(\underline{g})\leq 1$, that is, $|\lambda|\leq t$.
\end{proof}	
Let us now show how the invariant function $d$ determines the solvability of a Pick interpolation problem
\begin{theorem}\label{Schwarz-PickLemma}
	An $n$-point Pick interpolation problem $\Omega\ni z_j\mapsto w_j\in \D,1\leq j\leq n,$ is solvable if and only if one of the following two conditions holds.
	\begin{enumerate}
		\item $w_1=\ldots=w_n$.
		\item For at least one pair $(i,j)$, one has $w_i\neq w_j$ and $m(w_i,w_j)\leq d^\Omega_{(\underline{z}, {\underline{w} })} (z_i,z_j)$, where $\underline{w}=(w_1,\ldots,w_n)$.
	\end{enumerate}
When $\Omega=\D$, equality holds in $(2)$ if and only if there is a non-constant Blaschke product of degree at most $n-1$ solving the problem.
\end{theorem}
\begin{proof}
	Given an interpolation problem $\Omega\ni z_j\mapsto w_j\in \D,1\leq j\leq n,$ we see that the case where $(1)$ occurs is trivial. Therefore we assume that $w_i\neq w_j$ for some $i$ and $j$. By (\ref{RelationMinkoInvariant}) in Proposition \ref{PropsOfd}, $d_{\underline{w} }(z_i,z_j)=m(t w_i,t w_j)$ where $t=\frac{1}{\mu_{\mathscr{D}_n}(\underline{w})}$. Now, the problem is solvable if and only if $\underline{w}\in \overline{\sD_\Omega} (z_1,\ldots,z_n)$ which is equivalent to $t\geq 1$. Arguing as in the proof of Theorem \ref{D2Cara} gives the result.
	
	To see the statement involving $\D$, we use a result from Section \ref{Values_Of_d_For_D}. By Theorem \ref{BoundaryAndBlaschke}, $\partial\sD_\D(z_1,\ldots,z_n)\cap\D^n$ is the collection of points $(w_1,\ldots,w_n)$ which can be interpolated by non-constant Blaschke products of degree at mot $n-1$. Since $w_i\neq w_j$, by (5) in Proposition \ref{PropsOfd} we see that $m(w_i,w_j)= d_{\underline{w}} ^\D (z_i,z_j)$ if and only if $\mu_{\mathscr{D}_n}(\underline{w})=1$.
	 
\end{proof}

We want to call Theorem \ref{Schwarz-PickLemma} a \textit{generalized or multi-point Schwarz-Pick theorem}. Not only this gives a solvability criterion, we see that when the underlying domain $\Omega$ is the open unit disc $\D$, and we are dealing with $n$ points, $d$ is attained by a unique (up to a unimodular factor) Blaschke product of degree at most $n-1$. Thus, it generalizes the original Schwarz-Pick theorem. Interested reader may consult \cite{BeardonMinda} and \cite{EJLS2014} for more on multi-point Schwarz-Pick theorem.

Now, we introduce a generalized Lempert function. If we look at Lempert's theorem (\cite{J-P-Invariant}) and translate it in the language of Pick bodies, it says that if $\Omega$ is a convex domain, and $z_1,z_2\in \Omega$, then there is a map $\phi\in \mathcal{O}(\D,\Omega)$ mapping $\lambda_j\in \D$ to $z_j\in \Omega$ such that $\sD_\Omega(z_1,z_2)=\sD_\D (\lambda_1,\lambda_2)$. To author's knowledge, Lempert's theorem is the only tool available for finding the value of $c^*$ and we will try to use this idea to study $d$. 

Given any domain $\Omega$ and $n$ points $z_1,\ldots,z_n$ in $\Omega$, there is a holomorphic map $\phi:\D\rightarrow\Omega$ such that $z_1,\ldots,z_n\in \phi(\D)$. If $\phi:\lambda_j\mapsto z_j$, then $\sD_\Omega(z_1,\ldots z_n)\subset\sD_\D (\lambda_1,\ldots\lambda_n)$. It turns out that equality between these to domains occurs in a very special case and is not true in general. The next question is whether given an $\underline{\alpha}\in \C^n -\{\mathbf{0}\}$, the boundary point $\frac{\underline{\alpha}}{\mu_{\mathscr{D}_\Omega} (\underline{\alpha})}$ of $\sD_\Omega (z_1,\ldots,z_n)$ is a boundary point of some $\sD_\D (\lambda_1,\ldots,\lambda_n)$. To investigate this question we introduce the following.

For a given $\underline{\alpha}\in \C^n -\{\mathbf{0}\}$ and $1\leq i,j\leq n$, define

\begin{align}
	\delta_{(\underline{z},\underline{\alpha})} ^\Omega (z_i, z_j) &=\delta_{\underline{\alpha}} (z_i.z_j)\\
	&=inf\{d^\D _{(\underline{\lambda},\underline{\alpha})}(\lambda_i, \lambda_j):\phi(\lambda_k)=z_k,1\leq k\leq n,\,\,\text{for some}\,\,\phi\in\mathcal{O}(\D,\Omega)\},\notag
\end{align}
where $\underline{\lambda}=(\lambda_1,\ldots, \lambda_n)\in \D^n$ and $\underline{z}=(z_1,\ldots,z_n)$.

It is easy to see the following.
\begin{enumerate}
	\item For a fixed nonzero $\underline{\alpha}$, $\delta_{\underline{\alpha}}$ is contractible under holomorphic maps.
	\item $d_{(\underline{z},\underline{\alpha})} ^\Omega (z_i, z_j)\leq \delta_{(\underline{z},\underline{\alpha})} ^\Omega (z_i, z_j)$.
	\item For $n=2$ and $\underline{\alpha}=(\alpha_1,\alpha_2)$ with $\alpha_1\neq \alpha_2$, $\delta_{\underline{\alpha}}$ coincides with the Lempert function.
\end{enumerate}
We have the following result.
\begin{theorem}\label{EquivalenceForEquality}
	Suppose $\Omega$ is a Carath\'eodory hyperbolic domain. Fix mutually distinct $n$ points $z_1,\ldots,z_n$ in $\Omega$ and a nonzero $\underline{\alpha}=(\alpha_1,\ldots,\alpha_{n})\in\C^n$ with $\alpha_i \neq \alpha_j$ for some $i,j$. Then the following are equivalent.
	\begin{enumerate}
		\item\label{EqualityOfTwo} For some $i,j$ with $\alpha_i \neq \alpha_j$, $\delta_{(\underline{z},\underline{\alpha})} ^\Omega (z_i, z_j)$ is attained and  $d_{(\underline{z},\underline{\alpha})} ^\Omega (z_i, z_j)= \delta_{(\underline{z},\underline{\alpha})} ^\Omega (z_i, z_j)$.
		\item\label{CompositionBlaschke} There is a map $\phi\in \mathcal{O}(\D,\Omega)$ and a function $f\in\mathcal{O}(\Omega,\D)$ such that $(f(z_1),\ldots,f(z_n))\in\C\cdot \underline{\alpha}$, $z_1,\ldots,z_n\in\phi(\D)$ and $f\circ\phi$ is a non-constant Blaschke product of degree at most $n-1$.
	\end{enumerate}
	Moreover, if (\ref{EqualityOfTwo}) holds for one pair $(i,j)$, then it holds for all $i$ and $j$.
\end{theorem}
\begin{proof}
	$(1)\Rightarrow (2):$ By assumption, there exists a map $\phi\in \mathcal{O}(\D,\Omega)$ such that 
	\begin{align*}
		\delta_{(\underline{z},\underline{\alpha})} ^\Omega (z_i, z_j)=d^\D _{(\underline{\lambda},\underline{\alpha})} (\lambda_i,\lambda_j),
	\end{align*} where $\underline{\lambda}=(\lambda_1,\ldots,\lambda_n)$ and $\phi(\lambda_j)=z_j$. We consider a function $f\in\mathcal{O}(\Omega,\D)$ such that $(f(z_1),\ldots,f(z_n))\in\C\cdot \underline{\alpha}$ and $d_{(\underline{z},\underline{\alpha})} ^\Omega (z_i, z_j)=m(f(z_i),f(z_j))$. This gives us a function $f\circ \phi \in \mathcal{O}(\D,\D)$ such that $(f\circ \phi (\lambda_1),\ldots,f\circ \phi (\lambda_n))\in \C\cdot \underline{\alpha}$ and $d^\D _{(\underline{\lambda},\underline{\alpha})} (\lambda_i,\lambda_j)=m(f\circ \phi (\lambda_i),f\circ \phi (\lambda_j))$. By Theorem \ref{Schwarz-PickLemma}, $f\circ \phi$ is a Blaschke product of degree at most $n-1$.
	
	$(1)\Leftarrow (2):$ Suppose $\lambda_1,\ldots,\lambda_n\in \D$ are such that $\phi(\lambda_j)=z_j$ for all $j$. Since $f\circ \phi \in \mathcal{O}(\D,\D)$ is a Blaschke product of degree at most $n-1$, and $(f\circ \phi (\lambda_1),\ldots,f\circ \phi (\lambda_n))\in \C\cdot \underline{\alpha}$, we have $d^\D _{(\underline{\lambda},\underline{\alpha})} (\lambda_i,\lambda_j)=m(f\circ \phi (\lambda_i),f\circ \phi (\lambda_j))$. Hence the computation
	\begin{align}\label{Computation_For_Equality}
		\delta_{(\underline{z},\underline{\alpha})} ^\Omega (z_i, z_j)&\leq d^\D _{(\underline{\lambda},\underline{\alpha})} (\lambda_i,\lambda_j)\\
		&=m(f\circ \phi (\lambda_i),f\circ \phi (\lambda_j))\notag\\
		&=m(f(z_i),f(z_j))\notag\\
		&\leq d_{(\underline{z},\underline{\alpha})} ^\Omega (z_i, z_j)\,\,(\text{as}\,\, (f(z_1),\ldots,f(z_n))\in \mathbb{C}\cdot \underline{\alpha})\notag\\
		&\leq \delta_{(\underline{z},\underline{\alpha})} ^\Omega (z_i, z_j)\notag
	\end{align}
	proves the claims in (\ref{EqualityOfTwo}).

	The rest of the proof follows easily form the argument used in (\ref{Computation_For_Equality}).
\end{proof}
\begin{corollary}\label{CaraGeodesicCase}
	Suppose $\Omega$ is a Carath\'eodory hyperbolic domain, $z_1,\ldots,z_n\in \Omega$, and $\lambda_1,\ldots,\lambda_n\in \D$. 
	
	\begin{enumerate}
		\item If $z_1,\ldots,z_n$ lie on a Carath\'eodory geodesic, then for any nonzero $\underline{\alpha}$ 
		\begin{align*}
			d_{(\underline{z},\underline{\alpha})} ^\Omega (z_i, z_j)= \delta_{(\underline{z},\underline{\alpha})} ^\Omega (z_i, z_j)
		\end{align*}
		holds for all $i$ and $j$.
		\item If $\phi\in \mathcal{O}(\D,\Omega)$ sends $\lambda_j$ to $z_j$, then $\sD_\Omega(z_1,\ldots,z_n)=\sD_\D(\lambda_1,\ldots, \lambda_n)$ if and only if $\phi$ is a Carath\'eodory geodesic.
		\item When $\Omega$ is the polydisc $\D^m$, the equality $\sD_{\D^m}(z_1,\ldots,z_n)=\sD_\D(\lambda_1,\ldots, \lambda_n)$ is necessary and sufficient for the existence of a Carath\'eodory geodesic in $\D^m$ sending $\lambda_j$ to $z_j$.
	\end{enumerate}
\end{corollary}
\begin{proof}
	Since a Carath\'eodory geodesic comes with a left inverse, $(2)$ follows from Proposition 11.1.4. in \cite{J-P-Invariant} and this same result along with Theorem \ref{EquivalenceForEquality} implies $(1)$. $(3)$ follows from the facts that coordinate functions are in $\mathcal{O}(\D^m,D)$ and that any map in $\mathcal{O}(\D,\D^m)$ is a Cartesian product of $m$ functions from $\mathcal{O}(\D,\D)$.
\end{proof}

\section{Values of $d^\D$}\label{Values_Of_d_For_D}
Here we will describe the values of $d^\D _{\underline{\alpha}}$ for all nonzero $\underline{\alpha}$. First we give explicit values of $d^\D$ for the following $\underline{\alpha}$:
\begin{enumerate}
	\item $\underline{\alpha}=(0,0,\ldots,0,1)$.
	\item $\underline{\alpha}=(0,0,\ldots,0,\alpha_{n-1},\alpha_n)$ with $\alpha_{n-1}\alpha_n\neq 0$.
\end{enumerate}

\begin{proposition}
	If $\underline{\alpha}=(0,0,\ldots,0,1)$ then
	\begin{align*}
		d_{\underline{\alpha}} (z_i,z_j)&=0,\,\,\text{for}\,\,1\leq i,j\leq n-1,\\
		&=\prod_{l=1}^{n-1} \Big|\frac{z_l -z_n}{1- \overline{z_l }z_n}\Big|,\,\,\text{for}\,\,1\leq i\leq n-1\,\,\text{and}\,\,j=n.
			\end{align*}
\end{proposition}
\begin{proof}
	It is easy to see that for $1\leq i,j\leq n-1$, $d_{\underline{\alpha}} (z_i,z_j)=0$. Now suppose $j=n$ and $1\leq i\leq n-1$. There is an $f\in \mathcal{O}(\D,\D)$ such that $(f(z_1),\ldots,f(z_n))\in \C\cdot\underline{\alpha}$ and $d_{\underline{\alpha}} (z_i,z_n)=m(f(z_i),f(z_n))$. Since $f(z_1)=\cdots=f(z_{n-1})=0$, there is an $f_n\in\mathcal{O}(\D,\D)$ such that $$f(z)=f_n (z)\prod_{l=1}^{n-1}\Big(\frac{z_l -z}{1-\overline{z_l}z}\Big).$$
	Note that this $f$ satisfies $m(f(z_i),f(z_n))\leq \prod_{l=1}^{n-1} \Big|\frac{z_l -z_n}{1- \overline{z_l }z_n}\Big|$. Now the function $$\varphi(z)=\prod_{l=1}^{n-1}\Big(\frac{z_l -z}{1-\overline{z_l}z}\Big)$$ is in $\mathcal{O}(\D,\D)$ and $(\phi(z_1),\ldots,\phi(z_n))\in \C\cdot\alpha$. Hence 
	\begin{align*}
	m(\varphi(z_i),\varphi(z_n))\leq d_{\underline{\alpha}} (z_i,z_n)=m(f(z_i),f(z_n))\leq \prod_{l=1}^{n-1} \Big|\frac{z_l -z_n}{1- \overline{z_l }z_n}\Big|=m(\varphi(z_i),\varphi(z_n))
	\end{align*}
	and this completes the proof.
\end{proof}
Before proceeding with the other $\underline{\alpha}$, note that by (5) and (6) in Proposition \ref{PropsOfd}, for a given pair $(\beta_1,\beta_2)$ of distinct complex numbers and any two distinct $z_1,z_2$ in $\Omega$, $m(t\beta_1,t\beta_2)=c^* _\Omega(z_1,z_2)$ where $t=\frac{1}{\mu_{\mathscr{D}_\Omega (z_1,z_2)}(\beta_1,\beta_2)}$. 

For a given $\underline{\alpha}=(0,0,\ldots,0,\alpha_{n-1},\alpha_n)$ with $\alpha_{n-1}\alpha_n\neq 0$, suppose 
\begin{align}\label{DistinctComponents}
	\alpha^\prime _{n-1} =\frac{\alpha_{n-1}}{\prod_{l=1}^{n-2} \big(\frac{z_l -z_{n-1}}{1-\overline{z_l}z_{n-1}}\big)}\neq \frac{\alpha_{n}}{\prod_{l=1}^{n-2} \big(\frac{z_l -z_{n}}{1-\overline{z_l}z_{n}}\big)}=\alpha^\prime _{n}.
\end{align}
Then there is a $t$ as above, depending on $\underline{\alpha}$ and $(z_1,\ldots,z_n)$, such that $$m(t\alpha^\prime _{n-1},t\alpha^\prime _{n})=c^*_\D (z_{n-1},z_n)=m(z_{n-1},z_n).$$

By the Schwarz-Pick theorem, there is a $\varphi_t \in \autd$ such that $\varphi_t (z_{n-1})=t\alpha^\prime _{n-1}$ and $\varphi_t (z_{n})=t\alpha^\prime _{n}$. Now consider the function 
\begin{align}\label{FunctionForDiffntAlpha2}
	f_t (z)=\varphi_t (z)\prod_{l=1}^{n-2} \Big(\frac{z_l -z}{1-\overline{z_l}z}\Big).
\end{align}
This $f_t$ is in $\mathcal{O}(\D,\D)$ and $(f_t(z_1),\ldots,f_t(z_n))=(0,0,\ldots,0,t\alpha_{n-1},t\alpha_{n})\in \C\cdot\underline{\alpha}$.

\begin{proposition}
	Suppose we are given $\underline{\alpha}=(0,0,\ldots,0,\alpha_{n-1},\alpha_n)$ with $\alpha_{n-1}\alpha_n\neq 0$. Then we have the following.
	\begin{enumerate}
		\item $d_{\underline{\alpha}} (z_i,z_j)=0,\,\,\text{for}\,\,1\leq i,j\leq n-2.$
		\item If $\frac{\alpha_{n-1}}{\prod_{l=1}^{n-2} \big(\frac{z_l -z_{n-1}}{1-\overline{z_l}z_{n-1}}\big)}= \frac{\alpha_{n}}{\prod_{l=1}^{n-2} \big(\frac{z_l -z_{n}}{1-\overline{z_l}z_{n}}\big)}$, then 
		\begin{align*}
				d_{\underline{\alpha}} (z_i,z_j)&=|f_0 (z_j)|,\,\,\text{for}\,\,1\leq i\leq n-2,j=n-1,n,\\
			&=m(f_0 (z_{n-1}),f_0 (z_n)),\,\,\text{for}\,\,i=n-1,j=n
		\end{align*}
	where $f_0 (z)=\prod_{l=1}^{n-2}\big(\frac{z_l -z}{1-\overline{z_l}z}\big)$.
	\item  If $\frac{\alpha_{n-1}}{\prod_{l=1}^{n-2} \big(\frac{z_l -z_{n-1}}{1-\overline{z_l}z_{n-1}}\big)}\neq \frac{\alpha_{n}}{\prod_{l=1}^{n-2} \big(\frac{z_l -z_{n}}{1-\overline{z_l}z_{n}}\big)}$, then 
		\begin{align*}
		d_{\underline{\alpha}} (z_i,z_j)&=|f_t (z_j)|,\,\,\text{for}\,\,1\leq i\leq n-2,j=n-1,n,\\
		&=m(f_t (z_{n-1}),f_t (z_n)),\,\,\text{for}\,\,i=n-1,j=n
	\end{align*}
where $f_t (z)$ is given by (\ref{FunctionForDiffntAlpha2}).
	\end{enumerate}
\end{proposition}
\begin{proof}
	\begin{enumerate}
		\item This statement is trivial. 
		\item For proving the second statement, let us suppose that we have $\frac{\alpha_{n-1}}{f_0 (z_{n-1})}=\frac{\alpha_{n}}{f_0 (z_{n})}$, that is, $\C\cdot\underline{\alpha}=\C\cdot(f_0 (z_1),\ldots,f_0 (z_n))$. If $f\in \mathcal{O}(\D,\D)$ satisfies $(f(z_1),\ldots,f(z_n))\in \C\cdot(f_0 (z_1),\ldots,f_0 (z_n))$, then there is an $f_1\in \mathcal{O}(\D,\D)$ such that $f(z)=f_0 (z)f_1(z)$ for all $z\in \D$. Hence, for any $i\neq n-1,n$, we have $m(f(z_i),f(z_{n-1}))=|f(z_{n-1})|\leq |f_0 (z_{n-1})|$. Also $ |f_0 (z_{n-1})|\leq d_{\underline{\alpha}} (z_i,z_{n-1})$ and therefore we obtain that $d_{\underline{\alpha}} (z_i,z_{n-1})=|f_0 (z_{n-1})|$. Similarly we can show that $d_{\underline{\alpha}} (z_i,z_{n})=|f_0 (z_{n})|$. 
		
		Now suppose $f\in \mathcal{O}(\D,\D)$ satisfies $(f(z_1),\ldots,f(z_n))=\lambda (f_0 (z_1),\ldots,f_0 (z_n))$ for some $\lambda\in \C$ and $d_{\underline{\alpha}} (z_{n-1},z_{n})=m(f(z_{n-1}), f(z_n))$. So we have 
		\begin{align*}
		m(f_0(z_{n-1}), f_0(z_n))&\leq d_{\underline{\alpha}} (z_{n-1},z_{n})\\
		&=m(\lambda f_0(z_{n-1}),\lambda f_0(z_n))
		\end{align*}
		which implies $1\leq |\lambda|$. On the other hand, if $f_1\in \mathcal{O}(\D,\D)$ is the function that satisfies $f=f_0 f_1$, then we have 
		$$f_1 (z_{n-1})=\frac{f(z_{n-1})}{f_0 (z_{n-1})}=\lambda=\frac{f(z_{n})}{f_0 (z_{n})}=f_1 (z_{n}).$$
		Thus $|\lambda|\leq 1$. This concludes the proof.
		\item Let $\frac{\alpha_{n-1}}{\prod_{l=1}^{n-2} \big(\frac{z_l -z_{n-1}}{1-\overline{z_l}z_{n-1}}\big)}\neq \frac{\alpha_{n}}{\prod_{l=1}^{n-2} \big(\frac{z_l -z_{n}}{1-\overline{z_l}z_{n}}\big)}$ hold. We consider $f_t$ and $\varphi_t$ as constructed in (\ref{FunctionForDiffntAlpha2}), and $f_0$ as above. For $i\neq n-1,n$, there is an $f\in \mathcal{O}(\D,\D)$ such that $(f(z_1),\ldots,f(z_n))\in \C\cdot\underline{\alpha}$ and $d_{\underline{\alpha}} (z_i,z_{n-1})=m(f(z_i),f(z_{n-1}))=|f(z_{n-1})|$. Moreover, there is a $\lambda\in \C$ and an  $f_1 \in \mathcal{O}(\D,\D)$ such that $(f(z_1),\ldots,f(z_n))=\lambda (f_t(z_1),\ldots,f_t(z_n))$ and $f=f_0 f_1$. Clearly, $f_1 (z_{n-1})=\lambda \varphi_t(z_{n-1})$ (and similarly $f_1 (z_{n})=\lambda \varphi_t(z_{n})$). So we can write
		\begin{align*}
			m(\lambda\varphi_t (z_{n-1}),\lambda\varphi_t (z_{n}))&=m(f_1 (z_{n-1}),f_1 (z_{n}))\\
			&\leq m(z_{n-1},z_n)\\
			&=m(\varphi_t (z_{n-1}),\varphi_t (z_{n}))
		\end{align*}
		which suggests $|\lambda|\leq 1$. Also $|f_t (z_{n-1})|\leq d_{\underline{\alpha}} (z_i,z_{n-1})=|f(z_{n-1})|=|\lambda f_t (z_{n-1})|$ implies $1\leq |\lambda|$. Similarly, one can show that  $ d_{\underline{\alpha}} (z_i,z_{n})=|f_t (z_{n})|$.
		
		Now we show that $d_{\underline{\alpha}} (z_{n-1},z_n)=m(f_t (z_{n-1}),f_t (z_n))$. We find $g,g_1\in \mathcal{O}(\D,\D)$ and $\lambda\in \C$ such that $(g(z_1),\ldots,,g(z_n))=\lambda (f_t(z_1),\ldots,f_t(z_n))$,  $d_{\underline{\alpha}} (z_{n-1},z_n)=m(g (z_{n-1}),g(z_n))$ and $g=f_0 g_1$. This gives us $g_1(z_{n-1})=\lambda\varphi_t(z_{n-1})$ and $g_1(z_{n})=\lambda\varphi_t(z_{n})$. So we have
		\begin{align*}
			m(\lambda\varphi_t (z_{n-1}),\lambda\varphi_t (z_n))&=m(g_1 (z_{n-1}),g_1 (z_n))\\
			&\leq m(z_{n-1},z_n)\\
			&=m(\varphi_t (z_{n-1}),\varphi_t (z_n))
		\end{align*}
		and hence $|\lambda|\leq 1$. On the other hand \begin{align*}
			m(f_t (z_{n-1}),f_t (z_n))&\leq d_{\underline{\alpha}} (z_{n-1},z_n)\\
			&=m(g (z_{n-1}),g(z_n))\\
			&=m(\lambda f_t (z_{n-1}),\lambda f_t (z_n))
		\end{align*}
	implies $|\lambda|\geq 1$. The proof is now complete.
	\end{enumerate}
	\end{proof}
	Next let us describe the boundary of $\sD_n$. It can be shown that 
	\begin{align*}
		\partial\sD_n=&\{(w,\ldots,w):w\in \partial\D\}\\
		&\cup\big[\cup_{1\leq i<j\leq n}\{(w_1,\ldots,w_n)\in\D^n:w_i\neq w_j, m(w_i,w_j)=d_{\underline{w} }(z_i,z_j)\}\big].
	\end{align*}
	\begin{lemma}\label{BoundaryRelationOfDnAndDn-1}
	$(w_1,\ldots,w_n)\in \partial\sD_n\cap \D^n$ if and only if $$\Big(\frac{\varphi_{w_n}(w_1)}{\varphi_{z_n}(z_1)},\ldots,\frac{\varphi_{w_n}(w_{n-1})}{\varphi_{z_n}(z_{n-1})}\Big)\in\partial\sD_{n-1}$$ where $\varphi_u (v)=\frac{u-v}{1-\overline{u}v},\sD_n=\sD_\D (z_1,\ldots,z_n),$ and $\sD_{n-1}=\sD_\D (z_1,\ldots,z_{n-1})$.
	\end{lemma}
	\begin{proof}
		The proof follows from making the following observations:
		\begin{enumerate}
			\item $\partial\sD_n$ consists of the points that can be attained by functions of norm one but not by functions of norm strictly less than one.
			\item $\partial\sD_n$ is invariant under the action of the elements of $Aut(\sD_n)$ described in (\ref{AutosOfDnAssoctdWithAutD}).
			\item For any $f\in \mathcal{O}(\D,\D)$ with $f(u)=0$, there is an $f_1\in \mathcal{O}(\D,\D)$ such that $f(z)=\varphi_u (z)f_1 (z)$.
		\end{enumerate}
	\end{proof}
	
	The following two theorems relate the boundary $\partial\sD_n$ to solvable Pick interpolation problems with unique solutions.

\begin{theorem}\label{BoundaryAndBlaschke}
	$(w_1,\ldots,w_n)\in \partial\sD_n$ if and only if there is a finite Blaschke product $\varphi$ of degree at most $(n-1)$ such that $\varphi(z_j)=w_j,1\leq j\leq n$.
\end{theorem}
\begin{proof}
	($\Leftarrow$): We apply induction on $n$. Let $\varphi_{n-1}$ be a finite Blaschke product of degree at most $(n-1)$ and $\varphi_{n-1}(z_j)=w_j,1\leq j\leq n$.
	
	For $n=1$ we have $\varphi_0(z)=e^{i\theta}$, so $\varphi_0(z_1)\in\partial\sD_1=\partial\D$.
	
	For $n=2$, $\varphi_1$ is either the constant function with modulus one or is an element of $\autd$. In any case, $(\varphi_1(z_1),\varphi_1(z_2))\in \sD_2$.
	
	Let the statement be true for $n=m$. Suppose $\varphi_m$ is a Blaschke product of degree at most $m$ and $(w_1,\ldots,w_{m+1})=(\varphi_m(z_1),\ldots,\varphi_m(z_{m+1}))$. If the degree of $\varphi_m$ is zero, the result is clear. To see the other case, consider the function $\psi_m =\varphi_{w_{m+1}}\circ \varphi_m (z)$. Then $\psi_m$ is a Blaschke product of degree at most $m$ with a zero at $z_{m+1}$. Hence the function $\psi_{m-1}(z)=\frac{\psi_m(z)}{\varphi_{z_{m+1}}(z)}$ is a Blaschke product of degree at most $m-1$. By induction hypothesis $(\psi_{m-1}(z_1),\ldots,\psi_{m-1}(z_{m}))\in \partial \sD_m$, that is,
	$$\Big(\frac{\varphi_{w_{m+1}}(w_1)}{\varphi_{z_{m+1}}(z_1)},\ldots,\frac{\varphi_{w_{m+1}}(w_{m})}{\varphi_{z_{m+1}}(z_{m})}\Big)\in\partial\sD_{m}.$$
	Now Lemma \ref{BoundaryRelationOfDnAndDn-1} gives rest of the argument.
	
	($\Rightarrow$): We again apply induction on $n$. 
	
	For $n=1$ and $2$, the result is clear. Suppose that the statement is true for $n=m$ and let $(w_1,\ldots,w_{m+1})\in \partial\sD_{m+1}$. If $|w_{m+1}|=1$, the constant function $\varphi_m(z)=w_{m+1}$ works. For $|w_{m+1}|<1$, we have $(w_1,\ldots,w_{m+1})\in \partial\sD_{m+1}\cap\D^n$. So by Lemma \ref{BoundaryRelationOfDnAndDn-1} we obtain 
	$$\Big(\frac{\varphi_{w_{m+1}}(w_1)}{\varphi_{z_{m+1}}(z_1)},\ldots,\frac{\varphi_{w_{m+1}}(w_{m})}{\varphi_{z_{m+1}}(z_{m})}\Big)\in\partial\sD_{m}.$$
	By induction hypothesis there is a Blaschke product $\varphi_{m-1}$ of degree at most $m-1$ such that $\varphi_{m-1}(z_j)=\frac{\varphi_{w_{m+1}}(w_j)}{\varphi_{z_{m+1}}(z_j)},1\leq j\leq m$. We now take the Blaschke product $$\varphi_m (z)=\varphi_{w_{m+1}}(\varphi_{z_{m+1}}(z)\varphi_{m-1}(z))$$
	 which is of degree at most $m$ and satisfies $\varphi_m (z_j)=w_{j}, 1\leq j\leq m+1$. 
	 
	 This completes our proof.
\end{proof}
\begin{theorem}\label{BoundaryAndUniqueness}
	The solution to a solvable interpolation problem $\D\ni z_j\mapsto w_j\in \D, 1\leq j\leq n,$ is unique if and only if $(w_1,\ldots,w_n)\in \partial \sD_n$.
\end{theorem}
\begin{proof}
	The result is clear if $n=1$. Let the statement be true for $n=m$.
	
	($\Leftarrow$):  Let $(w_1,\ldots,w_{m+1})\in \partial\sD_{m+1}$. If $|w_{m+1}|=1$, then the unique solution is the constant function $\varphi(z)=w_{m+1}$. For the other case, Lemma \ref{BoundaryRelationOfDnAndDn-1} gives us 
	$$\Big(\frac{\varphi_{w_{m+1}}(w_1)}{\varphi_{z_{m+1}}(z_1)},\ldots,\frac{\varphi_{w_{m+1}}(w_{m})}{\varphi_{z_{m+1}}(z_{m})}\Big)\in\partial\sD_{m}.$$ 
	By induction hypothesis and Theorem \ref{BoundaryAndBlaschke}, there is a unique Blaschke product $\varphi_{m-1}$ of degree at most $m-1$ such that $\varphi_{m-1}$ takes $z_j$ to  $\frac{\varphi_{w_{m+1}}(w_j)}{\varphi_{z_{m+1}}(z_j)},1\leq j\leq m$. If $g\in \mathcal{O}(\D,\D)$ is a solution to the problem $z_j\mapsto w_j, 1\leq j\leq m+1$, then $\frac{\varphi_{w_{m+1}}(g(z))}{\varphi_{z_{m+1}}(z)}$ sends $z_j$ to  $\frac{\varphi_{w_{m+1}}(w_j)}{\varphi_{z_{m+1}}(z_j)},1\leq j\leq m$. Using the uniqueness of $\varphi_{m-1}$ we find that $g(z)=\varphi_{w_{m+1}}(\varphi_{z_{m+1}}(z)\varphi_{m-1}(z))$. So this $g$ is unique.
	
	($\Rightarrow$): Let the solution to the interpolation problem $z_j\mapsto w_j,1\leq j\leq m+1$, be unique and let $g$ be the solution. Clearly, the problem $z_j\mapsto \frac{\varphi_{w_{m+1}}(w_j)}{\varphi_{z_{m+1}}(z_j)},1\leq j\leq m,$ is solvable. If $h_0$ is a solution to this problem, then it is easy to see that the function $h(z)=\varphi_{w_{m+1}}(\varphi_{z_{m+1}}(z)h_0 (z))$ solves the interpolation problem $z_j\mapsto w_j,1\leq j\leq m+1,$ and hence, $h=g$. This gives us $h_0(z) =\frac{\varphi_{w_{m+1}}(g(z))}{\varphi_{z_{m+1}}(z)}$ (note that $\varphi_{w_{m+1}}(g(z))$ has a zero at $z_{m+1}$). So the uniqueness of $g$ passes onto $h_0$ and induction hypothesis together with Lemma \ref{BoundaryRelationOfDnAndDn-1} imply that $(w_1,\ldots,w_{m+1})\in \partial\sD_{m+1}$. 
	
	This completes the proof.
\end{proof}

The theorems above give us a way to relate the domain $\sD_n$ and the degree of the Blaschke product interpolating the boundary points of $\sD_n$.
\begin{corollary}
	Let $(w_1,\ldots,w_n)\in \partial \sD_n$ and $\varphi$ a Blaschke product sending $z_j$ to $w_j$, $1\leq j\leq n$. Then $\varphi$ is of degree $k-1$ if and only if $k$ is the least positive integer for which there are $w_{i_1},\ldots, w_{i_k}\in \{w_j:1\leq j\leq n\}$ such that $i_j \neq i_l$ for $j\neq l$ and
	\begin{align*}
		(w_{i_1},\ldots, w_{i_k})\in \partial\sD_\D (z_{i_1},\ldots, z_{i_k}).
	\end{align*}
\end{corollary}
\begin{proof}
($\Rightarrow$): Let $\varphi$ have degree $k-1$. Then using Theorem \ref{BoundaryAndBlaschke}, we see that for any $z_{i_1},\ldots, z_{i_k}\in \{z_1,\ldots, z_n\}$ we have
\begin{align*}
	(w_{i_1},\ldots, w_{i_k})=(\varphi(z_{i_1}),\ldots, \varphi(z_{i_k}))\in \partial\sD_\D (z_{i_1},\ldots, z_{i_k}).
\end{align*}
($\Leftarrow$): Let $k$ be the least positive integer  for which there are $w_{i_1},\ldots, w_{i_k}\in \{w_1,\ldots, w_n\}$ such that 
\begin{align*}
	(w_{i_1},\ldots, w_{i_k})\in \partial\sD_\D (z_{i_1},\ldots, z_{i_k}).
\end{align*}
Then by Theorem \ref{BoundaryAndBlaschke}, there is a Blaschke product $\psi$ of degree at most $k-1$ such that $\psi (z_{i_j})=w_{i_j}$, $1\leq j\leq k$. Note that $k$ is the least positive integer satisfying the above. If $\psi$ is a Blaschke product of degree less than $k-1$, then it contradicts with the minimality of $k$ (this follows from Theorem \ref{BoundaryAndBlaschke}). By Theorem \ref{BoundaryAndUniqueness}, $\psi$ is unique. Thus, $\psi$ is a Blaschke product of degree $k-1$ and $\psi=\varphi$. Thus the stated claim therefore holds true.
\end{proof}

Lastly, let us describe the values of $d_{\underline{\alpha}} ^\D$ for arbitrary $\underline{\alpha}$. We recall that if $\underline{\alpha}\in \mathbb{C}^n $ is a nonzero element and $t=\frac{1}{\mu_{\mathscr{D}_n} (\underline{\alpha})}$, then for any $i$ and $j$, $t\cdot\underline{\alpha}\in \partial\sD_n$ and $d_{\underline{\alpha}}(z_i,z_j)=m(t\alpha_i,t\alpha_j)$. Once one finds $t$, $d_{\underline{\alpha}}$ can easily be computed. We know that an interpolation problem $\D\ni z_j\mapsto w_j\in \D,1\leq i,j\leq n$, is solvable if and only if the matrix \begin{align*}
 	\mathcal{M}=\begin{pmatrix}
 		\frac{1-w_i\overline{w_j}}{1-z_i\overline{z_j}}
 	\end{pmatrix}_{1\leq i,j\leq n}
 \end{align*}
is positive semidefinite. Also, the solution is a unique if and only if $det(\mathcal{M})=0$. By Theorem \ref{BoundaryAndBlaschke} and \ref{BoundaryAndUniqueness} we can say that the positive number $t=\frac{1}{\mu_{\mathscr{D}_n} (\underline{\alpha})}$ is a root of the equation 
\begin{align}\label{EquationGivingMinkowski}
	det\begin{pmatrix}
		\frac{1-x^2 \alpha_i\overline{\alpha_j}}{1-z_i\overline{z_j}}
	\end{pmatrix}=0,
\end{align}
where the indeterminate is $x$. Let us now consider the following quantity
\begin{align}\label{ValueOfMinkowski}
	t=max\Big\{x\geq 0:x\,\,\text{is a root of}\,\,(\ref{EquationGivingMinkowski})\,\,\text{and}\,\,\begin{pmatrix}
		\frac{1-x^2 \alpha_i\overline{\alpha_j}}{1-z_i\overline{z_j}}
	\end{pmatrix}\geq \mathbf{0}\Big\}.
\end{align}

\begin{theorem}
	The quantity $t$ given by (\ref{ValueOfMinkowski}) is $\frac{1}{\mu_{\mathscr{D}_n} (\underline{\alpha})}$.
\end{theorem}
\begin{proof}
	We have 
	 \begin{align*}
			det\begin{pmatrix}
				\frac{1-t^2 \alpha_i\overline{\alpha_j}}{1-z_i\overline{z_j}}
			\end{pmatrix}=0\,\,\text{and}\,\,
			\begin{pmatrix}
				\frac{1-t^2 \alpha_i\overline{\alpha_j}}{1-z_i\overline{z_j}}
			\end{pmatrix}\geq \mathbf{0}.
		\end{align*}
	
	So the interpolation problem $z_j\mapsto t\alpha_j$ is solvable and its solution is unique. Theorem \ref{BoundaryAndUniqueness} implies that $t\cdot\underline{\alpha}\in \partial\sD_n$. Since $\partial\sD_n =\{\mu_{\mathscr{D}_n}=1\}$, the proof follows.
\end{proof}

\section{$\D^2$ and $\D^3$ with $n=3$}\label{D2_D3_description}
The case we discuss here involves only three point Pick interpolation problem on $\D^2$ and $\D^3$. For these two domains, we will show that statement $(2)$ in Theorem \ref{EquivalenceForEquality} holds and hence equality between two invariant functions introduced in Section \ref{Decription_of_d} takes place. We will prove the result for $\D^3$ because the proof for $\D^2$ is similar.

\begin{theorem}
	For any mutually distinct $z_1,z_2,z_3\in \D^3$ and any nonzero $\underline{\alpha}=(\alpha_1,\alpha_2,\alpha_3)$ in $\C^3$, the equality
	\begin{align*}
		d_{(\underline{z},\underline{\alpha})} ^{\D^3} (z_i, z_j)= \delta_{(\underline{z},\underline{\alpha})} ^{\D^3} (z_i, z_j)
	\end{align*}
	holds for all $i,j=1,2,3$.
\end{theorem}
\begin{proof}
	Without loss of generality we may assume $z_1$ is the origin and $\underline{\alpha}=(\alpha_1,\alpha_2,\alpha_3)\in \partial\sD_{\D^3}(z_1,z_2,z_3)$ (this is permitted by $(3)$ in Proposition \ref{PropsOfd}). Our goal is to find a $\phi\in\mathcal{O}(\D,\D^3)$ such that statement (\ref{CompositionBlaschke}) in Theorem \ref{EquivalenceForEquality} holds. With this goal in mind, we additionally assume that $\alpha_1=0$, for having constructed a Blaschke product of degree at most two, we can always compose it with an element of $\autd$ and find the required form. 
	
	When $z_1,z_2,z_3$ lie on a Carath\'eodory geodesic, the result follows from Corollary \ref{CaraGeodesicCase}. So we assume that $z_1,z_2,z_3$ do not lie on a Carath\'eodory geodesic. Since $(\alpha_1,\alpha_2,\alpha_3)\in \partial\sD_{\D^3}(z_1,z_2,z_3)$, the interpolation problem $z_j\mapsto \alpha_j$ is solvable and extremal. 
	
	If the problem is non-degenerate, by Theorem 1 in \cite{KosinskiPoly} and Theorem \ref{EquivalenceForEquality}, we have our result. If the problem is degenerate, we can assume that $c^* _{\D^3}(z_2,z_3)=m(\alpha_2,\alpha_3)$. Since the Carath\'eodory extremals in $\D^3$ are coordinate functions, we can also assume that $(\alpha_1,\alpha_2,\alpha_3)=(z_{11},z_{21},z_{31})$ and $m(z_{21},z_{31})\geq m(z_{2j},z_{3j})$ where the notation $z_j=(z_{j1},z_{j2},z_{j3})$ is used. Now consider the points $z_1 ^\prime =z_1, z_2 ^\prime =(z_{21},t_2 z_{22},t_3 z_{23}),  z_2 ^\prime =(z_{31},t_2 z_{32},t_3 z_{33})$ where $t_2,t_3\geq 1$ satisfy $m(z_{21},z_{31})= m(t_jz_{2j},t_jz_{3j})$. Next we pick a $\underline{\gamma}=(\gamma_1,\gamma_2,\gamma_3)\in \partial\sD_{\D^3}(z_1 ^\prime,z_2 ^\prime,z_3 ^\prime)$ such that the interpolation problem $z_j ^\prime\mapsto \gamma_j$ is a non-degenerate one (it is easy to see that $z_1 ^\prime\neq z_2 ^\prime\neq z_3 ^\prime\neq z_1 ^\prime$ and hence such a $\underline{\gamma}$ exists). By Lemma 3 in \cite{KosinskiPoly}, there exist $\varphi_1,\varphi_2,\varphi_3\in \autd$ and $\lambda_2,\lambda_3\in \D$ such that the map $\phi(\lambda)=(\lambda\varphi_1(\lambda),\lambda\varphi_2(\lambda),\lambda_3\varphi(\lambda))$ satisfies $\phi(0)=z_1 ^\prime$ and $\phi(\lambda_j)=z_j ^\prime$. Since $t_2,t_3\geq 1$, the map $\psi(\lambda)=(\lambda\varphi_1(\lambda),\frac{1}{t_2}\lambda\varphi_2(\lambda),\frac{1}{t_3}\lambda_3\varphi(\lambda))$ belongs to $\mathcal{O}(\D,\D^3)$, and clearly it passes through the original points $z_1,z_2,z_3$. Composing $\psi$ with the coordinate function delivers the result we need.
\end{proof}

When $n\geq 4$, the above result does not take place in general. If $\underline{\alpha}$ is an element of the standard basis of $\C^n$, the functions $d$ and $\delta$ coincides with the functions studied in \cite{Coman2000} and equality between them does not take place for $n\geq 4$ (Remark 8.2.17 in \cite{J-P-Invariant}).

\vspace{0.1in} \noindent\textbf{Acknowledgements:}
The author's research is supported by GACR (Czech grant agency) grant 22-15012J. The author thanks Prof. Wlodzimierz Zwonek and Dr. Anwoy Maitra for fruitful discussions.

\end{document}